\documentclass[notitlepage]{amsart}

\usepackage{amsfonts}
\usepackage[latin1]{inputenc}
\usepackage[all]{xy}

\newtheorem{theorem}{\textbf{Theorem}}
\newtheorem{corollary}[theorem]{\textbf{Corollary}}

\newtheorem{lemma}[theorem]{\textbf{Lemma}}

\newtheorem{proposition}[theorem]{\textbf{Proposition}}

\newcommand{\pn}{\par\noindent}

\newcommand{\mor}[3]{$\xymatrix@1@C=15pt{#3: #1\ar[r]& #2}$}
\newcommand{\iso}[3]{$\xymatrix@1@C=15pt{#3: #1\ar[r]^-{\cong}& #2}$}

\newcommand{\Hom}[1]{{\rm Hom}#1}
\newcommand{\Ext}[1]{{\rm Ext}#1}

\newcommand{\Ker}[1]{{\rm Ker}\ #1}

\newcommand{\findim}{{\rm fin.dim.\, }}
\renewcommand{\P}{\mathcal{P}^{<\infty}_\Lambda}
\newcommand{\Mod}{{\rm Mod\,}}
\renewcommand{\mod}{{\rm mod\,}}
\newcommand{\pd}{{\rm pd}}

\newcommand{\ind}{{\rm ind\,}}
\newcommand{\add}{{\rm add\,}}
\newcommand{\Add}{{\rm Add\,}}

\newcommand{\X}{\mathcal{X}}

\usepackage{latexsym,amssymb,amscd}
\usepackage{amsmath}
\usepackage[all]{xy}

\begin{document}
\sloppy
\bibliographystyle{plain}

\title{Approximations of injective modules and finitistic dimension}

\author{F. Huard and D. Smith}
\address{\pn Fran\c cois Huard; Department of mathematics,
Bishop's University. Sherbrooke, Qu\'ebec, Canada,  J1M1Z7.}
\email{fhuard@ubishops.ca}
\address{\pn David Smith; Department of mathematics,
Bishop's University. Sherbrooke, Qu\'ebec, Canada,  J1M1Z7.}
\email{dsmith@ubishops.ca}

\begin{abstract} Let $\Lambda$ be an artin algebra and let $\P$ the category of finitely generated right $\Lambda$-modules of finite projective dimension. We show that $\P$ is contravariantly finite in $\mod\Lambda$ if and only if the direct sum $E$ of the indecomposable Ext-injective modules in $\P$ form a tilting module in $\mod\Lambda$.  Moreover, we show that in this case $E$ coincides with the direct sum of the minimal right $\P$-approximations of the indecomposable $\Lambda$-injective modules and that the projective dimension of $E$ equal to the finitistic dimension of $\Lambda$.
\end{abstract}

\maketitle

Given an artin algebra $\Lambda$, we let $\P$ denote the class of all finitely generated right $\Lambda$-modules of finite projective dimension.  The finitistic dimension of $\Lambda$, $\findim\Lambda$, is then defined to be the supremum of the projective dimensions of the modules in $\P$.

It is still an open question whether $\findim\Lambda$ is always finite.  However, when $\P$ is contravariantly finite, it was shown by Auslander and Reiten that $\findim\Lambda$ is finite. They showed in \cite{AR} that under this hypothesis, the finitistic dimension of $\Lambda$ is equal to the maximum of the projective dimensions of the minimal right $\P$-approximations of the simple $\Lambda$-modules.  In this note, we show that this result remains valid when simples are replaced by indecomposable injectives. Moreover, and it is crucial to our proof, the direct sum of these $\P$-approximations of indecomposable injectives turns out to be a tilting module. This is not true in general for the $\P$-approximations of the simple $\Lambda$-modules.


Conversely, by using results from \cite{AHT, AS}, we show that if the direct sum $E$ of the indecomposable Ext-injective modules in $\P$ form a tilting module, then $\P$ is contravariantly finite and $E$ coincides with the direct sum of the minimal right $\P$-approximations of the indecomposable $\Lambda$-injective modules. This complements the results of \cite{HU} that were obtained for contravariantly finite and resolving subcategories of $\mathcal{P}_{\Lambda}^{\leq n}$, that is the category of all finitely generated right $\Lambda$-modules of projective dimension at most $n$, with $n\in\mathbb{N}$. 

  Let $\Mod\Lambda$ denote the category of all right $\Lambda$-modules and $\mod\Lambda$ be the category of finitely generated right $\Lambda$-modules. Let $\X$ be a subcategory of $\mod\Lambda$.  By a \emph{right $\X$-approximation} of a module $M\in\mod\Lambda$, we mean a morphism $f_M: X_M \to M$, with $X_M\in\X$, such that any morphism $f:X\to M$, with $X\in\X$, factors through $f_M$.  We say that $\X$ is \emph{contravariantly finite} if each $M\in\mod\Lambda$ admits a right $\X$-approximation.

A module $X\in \X$ is \emph{Ext-injective} in $\add\X$ (the class of finite direct sums of direct summands of modules in $\X$)  if $\Ext^1_\Lambda(-, X)\mid_{\X}=0$, and \emph{splitting injective} in $\X$ if each injective morphism $X\to Y$ with $Y$ in $\add\X$ splits.  We denote by $I_0(\X)$ the class of all indecomposable splitting injectives in $\X$.  The following is  easy.

\begin{lemma}\label{lem:easy}
Let $\X$ be a subcategory of $\mod\Lambda$ closed under extensions and cokernels of injections and let $X\in \X$.  Then $X$ is a splitting injective in $\X$ if and only if $X$ is Ext-injective in $\add\X$.
\end{lemma}

The following proposition is a particular case of \cite[Proposition 3.6]{AS}.
\begin{proposition}
Suppose that $\P$ is contravariantly finite. Let $I_1, I_2, \dots, I_n$ be a complete set of non-isomorphic indecomposable injective $\Lambda$-modules, and, for each $j=1, 2, \dots, n$, let $f_j:A_j\rightarrow I_j$ be the unique right minimal $\P$-approximations. Then
$$ I_0(\P)=\cup_{j=1}^n\ind A_j.$$
\noindent where $\ind A_j$ denotes the set of indecomposable direct summands of $A_j$.
\end{proposition}

Combining this proposition with the lemma, we get:

\begin{corollary}\label{cor:approx}
If $\P$ is contravariantly finite, then the indecomposable direct summands of the unique right minimal $\P$-approximations of the indecomposable injective $A$-modules coincide with the indecomposable Ext-injective modules in $\P$.
\end{corollary}

For a subcategory $\X$ of $\mod\Lambda$, we let $$\X^{\bot}=\{M \in\ \Mod\Lambda \ \mid \ \Ext^{i}_\Lambda(-, M)\mid_{\X}=0, \mbox{ for all } i>0\}.$$  In particular, the shifting lemma gives $$(\P)^{\bot}=\{M \in\ \Mod\Lambda \ \mid \ \Ext^{1}_\Lambda(-, M)\mid_{\P}=0\}.$$ The following result is the dual of \cite[Theorem 5.5 (b)]{AR}; see also \cite[Theorem 2.1]{S}. Recall that a subcategory $\X$ of $\mod\Lambda$ is said to be {\em resolving} if it contains the projective $\Lambda$-modules and it is closed for extensions and kernels of surjections.

\begin{proposition}
The map $\X \longmapsto \X\cap\X^{\bot}$ gives a one-to-one correspondence between contravariantly finite resolving subcategories $\X$ of $\P$ and equivalence classes of tilting $\Lambda$-modules in $\mod \Lambda$.
\end{proposition}

In the above proposition, a \emph{tilting module} is an object $T$ in $\mod\Lambda$ such that
\begin{enumerate}
\item[(i)] The projective dimension of $T$ is finite;
\item[(ii)] $T\in T^{\bot}$;
\item[(iii)] there exists a long exact sequence $0\rightarrow \Lambda_\Lambda \rightarrow T^0 \rightarrow\cdots\rightarrow T^n\rightarrow 0$, with $T^i\in {\rm add\,}T$ for all $i=0, 1, \dots, n$,
\end{enumerate}
and two tilting modules $T'$ and $T'$ are called \emph{equivalent} if $\add T=\add T'$.

\begin{corollary} \label{cor}
If $\P$ is contravariantly finite, then the indecomposable Ext-injectives in $\P$ form a finitely generated tilting $\Lambda$-module.
\end{corollary}
\begin{proof}
Indeed, we have
\[
\begin{array}{rcl}
(\P)\cap({\P})^{\bot} & = & \{M\in \P \ | \ \Ext^i_\Lambda(-, M)\mid_{\P}=0, \mbox{ for all } i>0\}\\ [2mm]
&=&  \{M\in \P \ | \ \Ext^1_\Lambda(-, M)\mid_{\P}=0\}\\ [2mm]
&=&  \{ \mbox{Ext-injectives in }\P\}.
\end{array}
\]
\end{proof}

We also need the following results from \cite{AHT}.   In what follows, $\Add \X$ denotes the class of direct summands of arbitrary direct sums of objects of $\X$.


\begin{lemma}\label{lem:ACT}\emph{\cite[Lemma 1.7]{AHT}} Let $T$ be a tilting module.  Then for each $M$ in $T^{\bot}$ there exists a short exact sequence
$\xymatrix@1{0\ar[r] & K_0 \ar[r] & T_0 \ar[r] & M \ar[r] & 0}$, where
$T_0$ is in $\Add\,T$ and $K_0$ is in $T^{\bot}$.
\end{lemma}

\begin{theorem}\label{lem:AHT} \emph{\cite[Theorems 2.6 and 4.2]{AHT}} Let $\Lambda$ be an artin algebra. The following conditions are equivalent:
\begin{enumerate}
\item[(a)] $\P$ is contravariantly finite in $\mod\Lambda$;
\item[(b)] there is a finitely presented tilting module $T$ such that $T^\bot=(\P)^{\bot}$.
\end{enumerate}
Moreover, in this case, the finitistic dimension of $\Lambda$ is equal to the projective dimension of $T$.
\end{theorem}

%
%
%
%

We can now show our main result.

\begin{theorem}  \label{thm:HS} Let $\Lambda$ be an artin algebra.  The following conditions are equivalent:
\begin{enumerate}
\item[(a)] $\P$ is contravariantly finite in $\mod\Lambda$;
\item[(b)] the direct sum $E$ of the  indecomposable Ext-injective modules in $\P$ form a tilting $\Lambda$-module.
\end{enumerate}
Moreover, in this case, $E$ is the direct sum of the isoclasses of minimal right $\P$-approximations of the indecomposable $\Lambda$-injectives modules, and the finitistic dimension of $\Lambda$ is equal to the projective dimension of $E$.
\end{theorem}
\begin{proof}
First, (a) implies (b) by Corollary \ref{cor}.  Thus, assume that the direct sum $E$ of the Ext-injective modules in $\P$ forms a tilting $\Lambda$-module. We will show that $E^\bot=(\P)^{\bot}$.

Since $E\in \P$, we have $(\P)^{\bot}\subseteq E^{\bot}$. Conversely, assume that $M\in E^{\bot}$.  We need to show that $\Ext^{i}_\Lambda(N, M)=0$ for all $N\in\P$ and all $i>0$.  So let $N\in \P$ be of projective dimension $n$.

Since $E$ is a tilting module and $M\in E^{\bot}$, it follows from Lemma \ref{lem:ACT} that there exists a short exact sequence
$$\xymatrix@1{0\ar[r] & K_0 \ar[r] & E_0 \ar[r] & M \ar[r] & 0},$$
where $E_0\in \Add E$ and $K_0\in E^{\bot}$. Inductively, we get an exact sequence
$$\xymatrix@1{0\ar[r] & K_{n-1} \ar[r] & E_{n-1} \ar[r]^-{f_{n-1}} & \cdots \ar[r] & E_1 \ar[r]^-{f_1} & E_0 \ar[r]^-{f_0} & M \ar[r] & 0},$$
where $E_i\in\Add E$ and $K_i = \Ker f_i$ lies in $E^{\bot}$ for all $i>0$.

By Corollary \ref{cor:approx} and \cite[Lemma 1.5]{AHT}, each $E_i$ is in $(\P)^{\bot}$.  
Therefore, applying the functor $\Hom_\Lambda(N, -)$ to each sequence
$$\xymatrix@1{0\ar[r] & K_i \ar[r] & E_i \ar[r]^{f_i} & K_{i-1} \ar[r] & 0},$$
(where $K_{-1}=M$) yields, for each $j>0$, an isomorphism
$$\Ext^j_\Lambda(N,K_{i-1}) \cong \Ext^{j+1}_\Lambda(N, K_i).$$
Therefore,
$$\Ext_\Lambda^j(N,M)\cong \Ext_\Lambda^{j+1}(N, K_0) \cong \cdots \cong \Ext_\Lambda^{j+n}(N, K_{n-1})=0$$
for all $j>0$ because $\pd\, N=n$.  Hence $M\in (\P)^{\bot}$.  Therefore $E^\bot=(\P)^{\bot}$ as claimed, and $\P$ is contravariantly finite by Theorem \ref{lem:AHT}. Thus (b) implies (a).

Finally, $E$ is the direct sum of the isoclasses of minimal right $\P$-approximations of the indecomposable $\Lambda$-injectives modules by Corollary \ref{cor:approx} and $\findim \Lambda = \pd\, E$ by Theorem \ref{lem:AHT}.
\end{proof}

\end{document}